\documentclass[12pt,reqno]{amsart}
\usepackage{amsmath, amsthm, amscd, amsfonts, amssymb, graphicx, color}
\allowdisplaybreaks

\makeatletter \oddsidemargin.9375in \evensidemargin \oddsidemargin
\allowdisplaybreaks
\newtheorem{theorem}{Theorem}[section]
\newtheorem{lemma}[theorem]{Lemma}
\newtheorem{proposition}[theorem]{Proposition}
\newtheorem{corollary}[theorem]{Corollary}
\theoremstyle{definition}

\newtheorem{example}[theorem]{Example}

\theoremstyle{remark}
\newtheorem{remark}[theorem]{Remark}
\numberwithin{equation}{section}

\begin{document}

\noindent {\footnotesize}

\title[On the semi-regular]
{On the semi-regular frames of translates}

\author[F. Valizadeh, H. Rahimi, R. A. Kamyabi Gol, F. Esmaeelzadeh]{F. Valizadeh$^{1}$, H. Rahimi$^{2, \ast}$, R. A. Kamyabi Gol$^{3}$, F. Esmaeelzadeh$^{4}$}
\address{$^{1}$ Department of Mathematics, Faculty of Science, Central Tehran
 Branch, Islamic Azad University, Tehran, Iran.;
\newline
} \email{fahimevalizade@yahoo.com}
\address{$^{2}$ Department of Mathematics, Faculty of Science, Central Tehran
 Branch, Islamic Azad University, Tehran, Iran.;
\newline
}  \email{rahimi@iauctb.ac.ir}

\address{$^{3}$ Department of Pure
Mathematics, Ferdowsi University of Mashhad, and Center of
Excellence in Analysis on Algebraic Structures (CEAAS), P.O.Box
1159-91775, Mashhad, Iran;
\newline
}  \email{kamyabi@um.ac.ir}

\address{$^{4}$ Department of Mathematics, Bojnourd Branch, Islamic Azad University,  Bojnourd, Iran;
\newline
} \email{esmaeelzadeh@bojnourdiau.ac.ir}


 \maketitle
\begin{abstract}
In this note, we fix a real invertible $d\times d$ matrix $\mathcal{A}$ and consider $\mathcal{A}\mathbb{Z}^d$ as an index set. For $f\in L^2(\mathbb{R}^d)$, let $\Phi^{\mathcal{A}}_{f}:=\frac{1}{|\det \mathcal{A}|}\sum_{k\in \mathbb{Z}^d}|\hat{f}(\mathcal{A}^T)^{-1}(\cdot+k)|^2$ be the periodization of $|\hat{f}|^2$. By using $\Phi^{\mathcal{A}}_{f}$, among other things, we characterize when the sequence $\tau_{\mathcal{A}}(f):=\{f(\cdot-\mathcal{A}k)\}_{k\in \mathbb{Z}^d}$ is a Bessel sequence, frame of translates, Riesz basis, or orthonormal basis. And finally, we construct an example, in which $\tau_{\mathcal{A}}(f)$ is a Parseval frame of translates, but not a Riesz sequence.
\end{abstract}

\section{Introduction}
Frames are generalization of orthonormal bases in Hilbert spaces. In 1946, Dennis Gabor, filled this gap and formulated a  fundamental approach to signal decomposition in terms of elementary signals \cite{12}. In 1952,  Duffin and Schaffer, presented some problems  in non-harmonic Fourier series, and frames  for Hilbert spaces \cite{8}. Later, Daubechies, Grossman and Mayer revived the study of frames  and applications \cite{7}. The main property of frames which make them useful is their redundancy. Many properties of frames make them useful in various applications  in mathematics, sciences  and engineering.\\
For a nice and comprehensive survey on various types of frames one may refer to, and the refrences therein \cite{5,6,13}.\\
Frames  of translates are an important class of frames that have a special structure. These frames are central in approximation, sampling,Gabor and wevelet theory, and were investigated in the context of general properties of shift invariant spaces by a number of authors including \cite{9,14}.\\
Frames of translates  are natural example of frame sequences.\\
Frame sequences are useful in cases where we are interested only in expansions in subspaces. For the  literature regarding  frames and frame sequences, one may refer to \cite{3,4,6,13}.\\
In this paper, we consider some kinds of semi-regular frames  of translates on the Hilbert space $L^2(\mathbb{R}^d)$. More precisely,  we investigate the frames of the form $\{T_{\mathcal{A}k}f\}_{k\in \mathbb{Z}^d}$ where $\mathcal{A}$ is a real  invertible  $d\times d$ matrix and $f\in L^2(\mathbb{R}^d)$,  and it is a frame for the  closed subspace generated by $\{T_{\mathcal{A}k}f\}_{k\in \mathbb{Z}^d}$.\\
This is a special case  of irregular frames of translates of the form $\{\lambda_{k}f\}_{k\in \mathbb{Z}^d}$ where $\{\lambda_{k}\}_{k\in \mathbb{Z}^d}\subseteq \mathbb{R}^d$ and $f\in L^2(\mathbb{R}^d)$\cite{1,2,3}.
\section{PRELIMINARIES AND NOTATION}
Let $\mathcal{H}$ be a separable Hilbert space with inner product $\langle\cdot ,\cdot\rangle$.
A sequence $\{f_k\}_{k=1}^{\infty}$ is called a basis for $\mathcal{H}$,  if for every $f\in \mathcal{H}$ there is a  unique sequence of scalars $\{c_k\}_{k=1}^{\infty}$ such that $f=\sum_{k=1}^{\infty} c_kf_k $,
and a sequence $\{e_k\}_{k=1}^{\infty} \subseteq \mathcal{H}$ is an orthonormal system if $\langle {e_k},{e_j}\rangle={\delta}_{k,j}$. An orthonormal basis is an orthonormal system $\{e_k\}_{k=1}^{\infty}$ which is a basis for $\mathcal{H}$. \\
A Riesz basis for $\mathcal{H}$ is a family of the form $\{Ue_k\}_{k=1}^{\infty}$, where $\{e_k\}_{k=1}^{\infty}$ is an orthonormal basis for $\mathcal{H}$ and $U:\mathcal{H}\longrightarrow \mathcal{H}$ is a bounded bijective operator. A Riesz basis is actually a basis.
In fact, a basis $\{f_k\}_{k=1}^{\infty}$ in $\mathcal{H}$ is a Riesz basis, if $\sum_{k=1}^{\infty}c_kf_k$ converges in  $\mathcal{H}$ only when $\sum_{k=1}^{\infty}|c_k|^2<\infty$.\\
 A sequence $\{f_k\}_{k=1}^{\infty}$  of   elements  in  $\mathcal{H}$  is a frame for $\mathcal{H}$ if there exist constants ${A},{B} > 0$ such that,
\begin{equation}\label{eqkh2.1}
A\|f\|^2\leqslant\sum_{k=1}^{\infty}|\langle\ f,f_k\rangle|^2\leqslant B\|f\|^2, \forall f\in {\mathcal{H}}
\end{equation}
The numbers  $A,B$ are called frame bounds, they are not unique. A frame with bounds  $A=B=1$,  is called a Parseval frame. A sequence $\{f_k\}\subseteq \mathcal{H}$ is called a Bessel sequence if the right hand side of inequality of \eqref{eqkh2.1} holds.\\
Note that, an orthonormal system $\{e_k\}_{k=1}^{\infty}$ is a Bessel sequence.\\
A frame is excat,  if it ceases to be a frame whenever any single element is deleted from the sequence. A frame which is not exact is called  an overcomplete frame.\\
We recall that, every orthonormal basis is an exact Parseval frame and conversely,  every exact Parseval frame is an orthonormal basis.\\
A sequence $\{f_k\}_{k=1}^{\infty}$ is called a Riesz sequence if there exist $A,B>0$ such that
\[
A\sum_{k=1}^{\infty}|{c_k}|^2\leqslant \|\sum_{k=1}^{\infty}{c_k f_k}\|^2\leqslant B\sum_{k=1}^{\infty}|c_k|^2,
\]
for all finite sequence $\{c_k\}_{k=1}^{\infty}\subseteq \mathbb{C}$. A Riesz sequence $\{f_k\}_{k=1}^{\infty}$ in Hilbert space $\mathcal{H}$,  is a Riesz basis for the Hilbert space $\overline{span}\{f_k\}_{k=1}^{\infty}$, which might just be a subspace of $\mathcal{H}$.\\
Note that, $\{f_k\}_{k=1}^{\infty}$  is a frame sequence if it is a frame for $\overline{span}\{f_k\}_{k=1}^{\infty}$ \cite{6,10,11,13}.\\
Frame sequences and Riesz sequences are useful consepts in cases where we only obtain expansions in subspaces. \\
For $\mathcal{A}\in GL_d(\mathbb{R})$, let $\Lambda=\mathcal{A}\mathbb{Z}^d$ and $Q_{\mathcal{A}}=\mathcal{A}[0,1)^d$, then $\mathbb{R}^d=\bigcup_{\lambda \in \Lambda}(\lambda+Q_{\mathcal{A}})$ such that $(\lambda+Q_{\mathcal{A}})\bigcap (\lambda^{'}+Q_{\mathcal{A}})=\emptyset$ if and only if $\lambda\neq \lambda^{'}$.\\
A function $f :\mathbb{R}^d \longrightarrow \mathbb{C}$ is called $\Lambda$-periodic if $f(x+\lambda)=f(x)$ for all  $x \in \mathbb{R}^d$ and $\lambda \in \Lambda$.\\
Note that, $f$ is $\Lambda$-periodic if and only if $fo\mathcal{A}$ is periodic. We denote $fo\mathcal{A}$ by $f_{\mathcal{A}}$.\\
 Recall that for  $f \in L^1(\mathbf{R}^d)$, the  Fourier transform of $f$ is\\
$\mathcal{F}f(\xi) :=\hat{f}(\xi) =\int_{\mathbb{R}^d} f(x) e^{-2\pi i \xi \cdot x}dx$ ,   $\xi \in \mathbb{R}^d$\\
and the inversion Fourier transform is $\check{f} (x)=\hat{f}(-x)$.\\
As usual the definition of Fourier transform extends to a unitary operator $f \longmapsto \hat{f}$ on $L^2(\mathbb{R}^d)$ , known as Plancharel theorem.\\
Note that, $\hat{(f_{\mathcal{A}})}(\xi)=\frac{1}{|\det \mathcal{A}|}\hat{f}((\mathcal{A}^T)^{-1}\xi)$.( see \cite{11}),
and the Fourier series of $f_{\mathcal{A}}$ is
\begin{align*}
\sum_{k \in \mathbb{Z}^d}\hat{f_{\mathcal{A}}}(k)e^{-2 \pi i k\cdot x}&=\frac{1}{|\det \mathcal{A}|}\sum_{k\in \mathbb{Z}^d}\hat{f}((\mathcal{A}^T)^{-1}k)e^{-2 \pi i k\cdot x}\\
&=\frac{1}{|\det \mathcal{A}|}\sum_{\lambda \in \Lambda^{\perp}}\hat{f}(\lambda) e^{-2\pi i\lambda \cdot x},
\end{align*}
where $\Lambda^{\perp}=(\mathcal{A}^T)^{-1}\mathbb{Z}^d$.\\
For  $\psi \in L^2(Q_{\mathcal{A}})$, put
$
\varepsilon_{\mathcal{A}}(\psi)=\{(E_{\mathcal{A}k}\psi)(\gamma)\}_{k\in \mathbb{Z}^d}
$
where $(E_{\mathcal{A}k}\psi)(x)=e^{2\pi i \mathcal{A}k\cdot x}\psi(x)$
for $x\in Q_{\mathcal{A}}$ and
$
N_{\psi}^{\mathcal{A}}=\{ x\in Q_{\mathcal{A}} :\psi (x)=0\}
$.
For any function $f$ on $\mathbb{R}^d$  and $k\in \mathbb{Z}^d$, the translation of $f$ by $k$,  is $T_{\mathcal{A}k} f(x)= f(x-\mathcal{A}k)$ and the set of all translations of $f$ is denote by
$
\tau_{\mathcal{A}} (f) =\{T_{\mathcal{A}k}f\} _{k\in \mathbb{Z}^d}
$ which
 is called a system of translates of $f$.
And
$
\mathcal{F}(\tau_{\mathcal{A}}f):=\{(T_{\mathcal{A}k} f)^{\wedge}\}_{k\in \mathbb{Z}^d}=\{E_{-\mathcal{A}k} \hat{f}\}_{k\in \mathbb{Z}^d}
$
For  $f \in L^2(\mathbb{R}^d)$ and $\mathcal{A}\in GL_d(\mathbb{R})$.\\For $f\in L^2(\mathbb{R}^d)$, we define
\[
\Phi_{f}^{\mathcal{A}}(\gamma):=\frac{1}{|\det \mathcal{A}|}\sum_{k \in \mathbb{Z}^d}|\hat{f}\big((\mathcal{A}^T)^{-1}(\gamma+k)\big)|^2, \quad \gamma \in \mathbb{R}^d.
\]
We call $\Phi_{f}^{\mathcal{A}}$ the periodization of $|\hat{f}|^2$.
It is worthwhile to mention for  $f\in L^1(\mathbb{R}^d)$, the series $\sum_{k\in \mathbb{Z}^d} T_{\mathcal{A}k} f$ converges pointwise a.e., and in $L^1(Q_{\mathcal{A}})$ to a function $Pf$ such that $\|Pf\|_{L^1(Q_{\mathcal{A}})}\leq \| f\|_{L^1(\mathbb{R}^d)}$. Moreover for $\gamma\in \mathbb{Z}^d$, $(Pf)^{\widehat{\ \ }}(\gamma)$ (Fourier transform on $Q_{\mathcal{A}}$) equals $\hat{f}(\gamma)$ (Fourier transform on $\mathbb{R}^d$) see \cite{10}.  For  $f\in L^2(\mathbb{R}^d)$, one can check easily that,
\[
\|\Phi_{f}^{\mathcal{A}}\|_{L^{1}(Q_{\mathcal{A}})}=\frac{1}{|\det \mathcal{A}|}\|f\|^2_{L^2(\mathbb{R}^d)}.
\]
Consequently,
\[
(\Phi_{f}^{\mathcal{A}})^{\frac{1}{2}}(\gamma)=\frac{1}{\sqrt{|\det \mathcal{A}|}}\Big(\sum_{k \in \mathbb{Z}^d}|\hat{f}\big((\mathcal{A}^T)^{-1}(\gamma+k)\big)|^2\Big)^{\frac{1}{2}},
\]
belongs to $L^2(Q_{\mathcal{A}})$, and
$
\|(\Phi_{f}^{\mathcal{A}}) ^{\frac12} \| _{L^2(Q_{\mathcal{A}})}=\frac{1}{\sqrt{|\det \mathcal{A}|}}\|f\| _{L^2(\mathbb{R}^d)}.
$
Therefore $f\mapsto (\Phi _{f}^{\mathcal{A}})^{\frac{1}{2}}$ is a norm-preserving map of $L^2(\mathbb{R}^d)$ into $L^2(Q_{\mathcal{A}})$. In this paper,  among  other things, we use  $\Phi_{f}^{\mathcal{A}}$ for characterizing when $\tau_{\mathcal{A}}(f)$ is a Bessel sequence, frame of translates, Riesz basis or an  orthonormal basis.
\section{MAIN RESULTS}
Frames of translates are natural examples  of frame sequences. Frames  of translates play important roles in many areas, including  wavelet theory and reconstruction of signals  from sample values. \\
We recall that, for $f\in L^2(\mathbb{R}^d)$ if  $\tau_{\mathcal{A}}(f)$ is a Bessel sequence, then for any $\{c_k\}_{k \in \mathbb{Z}^d}\in l^2(\mathbb{Z}^d)$, $\sum_{k\in \mathbb{Z}^d} c_kT_{\mathcal{A}k}(f)$ converges in $L^2(\mathbb{R}^d)$ and $\sum_{k \in \mathbb{Z}^d} c_k E_{-\mathcal{A}k}$ converges in $L^2[0,1)^d$, where $E_{-\mathcal{A}k}(x)=e^{-2\pi i \mathcal{A}k\cdot x}$ , $x\in \mathbb{R}^d$.\\

Now for $f\in L^1(\mathbb{R}^d)$,let $\psi_{f}^{\mathcal{A}}(\gamma)=\sum_{k \in \mathbb{Z}^d} f(\gamma+\mathcal{A}k)$  the $\Lambda$-periodization of $f$,  we show that $\psi_{f}^{\mathcal{A}}$ is absolutely convergence in $L^1(Q_{\mathcal{A}})$ and $\|\psi_{f}^{\mathcal{A}}\|_{L^1(Q_{\mathcal{A}})}=\|f\|_{L^1(\mathbb{R}^d)}$.
\begin{lemma}
Let $f\in L^1(\mathbb{R}^d)$. With notations as above, the series defining $\psi_{f}^{\mathcal{A}}$ converges absolutely in $L^1(Q_{\mathcal{A}})$ and $\int_{Q_{\mathcal{A}}} \psi_{f}^{\mathcal{A}} (\gamma)d\gamma=\int_{\mathbb{R}^d} f(\gamma)d\gamma$.
\end{lemma}
\begin{proof}
Let $f\in L^1(\mathbb{R}^d)$, $\mathcal{A}\in GL_{d}(\mathbb{R})$ and put $\psi_{f}^{\mathcal{A}}(\gamma)=\sum_{k \in \mathbb{Z}^d}f(\gamma+\mathcal{A}k)$,  then using the fact $\psi_{f}^{\mathcal{A}}$ is $\Lambda$-periodic, we have
\[
\int_{Q_{\mathcal{A}}}|\psi(\gamma)|d\gamma\leq
\int_{Q_{\mathcal{A}}} \sum_{k\in \mathbb{Z}^d}|f(\gamma+\mathcal{A}k)|d\gamma
 =\sum_{k\in \mathbb{Z}^d}\int_{Q_{\mathcal{A}}}|f(\gamma+\mathcal{A}k)|d\gamma
\]
\[
 =\int_{\bigcup_{k\in \mathbb{Z}^d}Q_{\mathcal{A}}+\mathcal{A}k } |f(\gamma+\mathcal{A}k)|d\gamma=\int_{\mathbb{R}^d} |f(\gamma+\mathcal{A}k)|d\gamma = \|f\|_1<\infty,
\]
 which is finite, because  $f\in L^1(\mathbb{R}^d)$. This implies that  $\psi_{f}^{\mathcal{A}}(\gamma)=\sum_{k \in \mathbb{Z}^d}f(\gamma+\mathcal{A}k)\in L^1(Q_{\mathcal{A}})$ and $\|\psi_{f}^{\mathcal{A}}\|_{L^1(Q_{\mathcal{A}})}\leqslant \|f\|_{1}$.
  Now by  using the  dominated convergence theorem, since
  \[
  \psi_N(\gamma)=\sum_{|k|\leq N} f(\gamma+\mathcal{A}k) \to \sum_{k\in \mathbb{Z}^d} f(\gamma+\mathcal{A}k) \quad \text{as}\  N \to \infty
  \]
where $|k|=\sum_{1}^n k_j$,
  and
  \[
  |\psi_N(\gamma)|\leq \sum_{k\in \mathbb{Z}^d} |f(\gamma+\mathcal{A}k)|\in L^1 (Q_{\mathcal{A}}).
  \]
  Then
  \[
  \int_{Q_{\mathcal{A}}} \psi_N(\gamma)d\gamma\to \int_{Q_{\mathcal{A}}} \sum_{k\in \mathbb{Z}^d} f(\gamma+\mathcal{A}k)d\gamma.
  \]
  And,
  \[
  \int_{Q_{\mathcal{A}}} \sum_{k\in \mathbb{Z}^d} f(\gamma+\mathcal{A}k) d\gamma =\sum_{k\in \mathbb{Z}^d} \int_{Q_{\mathcal{A}}} f(\gamma+\mathcal{A}k)d\gamma
  \]
  \[
  =\int_{\bigcup_{k\in \mathbb{Z}^d}Q_{\mathcal{A}}+\mathcal{A}k }\psi (\gamma )d\gamma =\int_{\mathbb{R}^d}f(\gamma) d\gamma.
  \]
\end{proof}
Note that, $\varepsilon_{\mathcal{A}}(\psi)$ is a Bessel sequence in $L^2(Q_{\mathcal{A}})$ if and only if $\psi_{\mathcal{A}} \in L^{\infty}(Q_{\mathcal{A}})$. In this case $|\psi(\gamma)|^2\leqslant B$ , a.e., where $B$ is a Bessel bound.\\
Now we show that $\tau_{\mathcal{A}}(f)$ is a Bessel sequence if and only if $\Phi_{f}^{\mathcal{A}}$ is bounded above  in frequency space.
\begin{proposition}\label{prkh 3.3}
Let $f \in L^2(\mathbb{R}^d)$ and  $B>0$,  then
$\tau_{\mathcal{A}}(f)$ is a Bessel sequence with bound B if and only if,\\
$\Phi_{f}^{\mathcal{A}}(\gamma)\leqslant B$, a.e. $\gamma\in Q_{\mathcal{A}}$.
\end{proposition}
\begin{proof}
Given a finite sequence $\{c_k\}_{k\in \mathbb{Z}^d}\subseteq \mathbb{C}$ and put
\[
\psi(\gamma)=\sum_{k\in \mathbb{Z}^d} c_{k} e^{-2\pi i (\mathcal{A}^T)^{-1}k\cdot \gamma}.
\]
Note that $\psi\in L^2(Q_{(\mathcal{A}^T)^{-1}})$, and  we have
\begin{align*}
\|T\{c_k\}_{k\in \mathbb{Z}^d} \|^2 &=\|\sum_{k\in \mathbb{Z}^d} c_k T_{(\mathcal{A}^T)^{-1}k}f\|^2\cr
&=\Big\|\mathcal{F}( \sum_{k\in \mathbb{Z}^d} c_kT_{(\mathcal{A}^T)^{-1}k}f)\Big\|^2\cr
&=\big\|\sum_{k\in \mathbb{Z}^d}c_kE_{-(\mathcal{A}^T)^{-1}k}\hat{f}\big\|^2\cr
&=\int_{\mathbb{R}^d} \Big|\sum_{k\in \mathbb{Z}^d} c_{k}e^{-2\pi i(\mathcal{A}^T)^{-1}k\cdot \gamma} \hat{f}(\gamma)\Big|^2 d\gamma\cr
&=\int_{\hat{\mathbb{R}^d}}| \psi (\gamma)|^2 |\hat{f}(\gamma)|^2d\gamma\cr
&=\frac{1}{|\det \mathcal{A}|}\int_{[0,1]^d}|\psi(\gamma)|^2\sum_{k \in \mathbb{Z}^d } |\hat{f}(\mathcal{A}^T)^{-1}(\gamma+k)|^2 d\gamma\cr
&=\int_{[0,1]^d}|\psi(\gamma)|^2\Phi_{f}^{\mathcal{A}}(\gamma)d\gamma.
\end{align*}
If $\Phi_{f}^{\mathcal{A}}(\gamma)\leqslant B$ for a.e. $\gamma\in \mathbb{R}^d$, it follows that,\\
$\|T\{c_{k}\}_{k \in \mathbb{Z}^d}\|^2\leqslant B\int_{[0,1]^d}|\psi(\gamma)|^2d\gamma
=B \sum_{k \in \mathbb{Z}^d}|c_{k}|^2$.
\end{proof}
Now we show that, $\tau_{\mathcal{A}}(f)$ is an  orthonormal sequence if and only if  $\Phi_{f}^{\mathcal{A}}(\gamma)=1$, a.e..
\begin{proposition}
Let $f  \in L^2(\mathbb{R}^d)$, then $\tau_{\mathcal{A}}(f)$ is an orthonormal sequence if and only if
$
\Phi_{f}^{\mathcal{A}} (\gamma) =1$,
a.e., $\gamma \in Q_{\mathcal{A}}$.
\end{proposition}
\begin{proof}
We have to show that for $k_1,k_2\in \Lambda$, $\langle T_{\mathcal{A}k_1}f, T_{\mathcal{A}k_2}f\rangle=\delta_{\mathcal{A}k_1,\mathcal{A}k_2}$,  if and only if $\Phi_{f}^{ \mathcal{A}}=1$ a.e., $\gamma\in \mathbb{R}^d$\\
Using the Plancharel theorem and Weil's formula, we have
\begin{align}
\langle T_{\mathcal{A}k_1}f, T_{\mathcal{A}k_2}f\rangle&=\langle f, T_{\mathcal{A}(k_1-k_2)}f\rangle\cr
&=\langle \hat{f}, \widehat{T_{\mathcal{A}(k_1-k_2)}f}\rangle\cr
&=\int_{\mathbb{R}^d}|\hat{f}(\gamma)|^2 \cdot e^{-2\pi i\mathcal{A}(k_1-k_2)\cdot \gamma}d\gamma\cr
&=\frac{1}{|\det \mathcal{A}|}\int_{Q_{\mathcal{A}}}\sum_{\lambda\in \mathbb{Z}^d} \big|\hat{f}\big((\mathcal{A}^T)^{-1}(\gamma+\lambda)\big)|^2 e^{-2\pi i (\mathcal{A}^T)^{-1}k\cdot \gamma}d\gamma\cr
&=\int_{Q_{\mathcal{A}}}\Phi_{f}^{\mathcal{A}}(\gamma)\cdot e^{-2\pi i(\mathcal{A}^T)^{-1} k \cdot \gamma}d\gamma. \label{eqkh3.1}
\end{align}
Pontryagin  duality theorem and \cite[Proposition $4\!\cdot\!3$]{10}, imply that $\Lambda$ is an orthonormal basis for $L^2(Q_{\mathcal{A}})$. Now \eqref{eqkh3.1} completes the proof.
\end{proof}
We continue to determine conditions, such that $\varepsilon_{\mathcal{A}}(\psi)$ (Theorem \ref{thkh3..3.5}) and $\tau_{\mathcal{A}}(f)$ (Theorem \ref{thkh3.7}) are frames. But first we need to prove a lemma, that shows
\[
\overline{span}\varepsilon_{\mathcal{A}}(\psi)=\{f\in L^2(Q_{\mathcal{A}}):f=0\ \text{a.e.  on} \ N_{\psi}^ {\mathcal{A}}\}.
\]
Then by using this, we specify conditions for $\varepsilon_{\mathcal{A}}(\psi)$ such that, it is a frame for $L^2(Q_{\mathcal{A}})$ and $\{f\in L^2(Q_{\mathcal{A}}):f=0\ \text{a.e.} \ \text{on} \ N_{\psi}^ {\mathcal{A}}\}$. Finally, in Theorem \ref{thkh3.7}., we determine such conditions that the set of translations of $f$, that is $\tau_{\mathcal{A}}(f)$,  is a frame.
\begin{lemma}\label{lekh3.4}
Let $\psi \in L^2(Q_{\mathcal{A}})$ be bounded, then $\overline{span}\varepsilon_{\mathcal{A}}(\psi)=\{f\in L^2(Q_{\mathcal{A}}):f=0\ \text{a.e.  on} \ N_{\psi}^ {\mathcal{A}}\}$.
\end{lemma}
\begin{proof}
It is easy to check that  $\{f\in L^2(Q_{\mathcal{A}}):f=0\ \text{a.e.} \ \text{on} \ N_{\psi}^{ \mathcal{A}}\}$ is a closed subspace, and it is clear that
\[
span\varepsilon_{\mathcal{A}}(\psi)\subseteq \{f\in L^2(Q_{\mathcal{A}}):f=0\ \text{a.e.  on} \ N_{\psi}^{\mathcal{A}}\}.
\]. To show that
$\overline{span}\varepsilon_{\mathcal{A}}(\psi)=\{f\in L^2(Q_{\mathcal{A}}):f=0\ \text{a.e.  on} \ N_{\psi}^{ \mathcal{A}}\}$, \\
suppose that $f\in \{f\in L^2(Q_{\mathcal{A}}):f=0\ \text{a.e.  on} \ N_{\psi}^{ \mathcal{A}}\}$ satisfies  $\langle f, \psi  e_k \rangle =0$ for every $k\in \mathbb{Z}^d$. Since $\psi $ is bounded, $f\overline{\psi}\in L^2(Q_{\mathcal{A}})$ and $Q_{(\mathcal{A}^T)^{-1}}$ is an orthonormal basis for $L^2(Q_{\mathcal{A}})$, (See \cite{10})  we have  $\langle f \overline{\psi}, e_k\rangle =\langle f, \psi  e_k\rangle=0$
for $k\in \mathbb{Z}^d$, therefore  $f\overline{\psi}=0$ a.e., and since $f\in \{f\in L^2(Q_{\mathcal{A}}):f=0\ \text{a.e.  on} \ N_{\psi}^{ \mathcal{A}}\}$ we have  $f=0$ a.e.. Hence $\varepsilon_{\mathcal{A}}(\psi)$ is complete in the  set $\{f\in L^2(Q_{\mathcal{A}}):f=0\ \text{a.e.  on} \ N_{\psi}^{ \mathcal{A}}\}$, and we are done.
\end{proof}
\begin{theorem}\label{thkh3..3.5}
Given $\psi \in L^2(Q_{\mathcal{A}})$, then $\varepsilon_{\mathcal{A}}(\psi)$ is a frame sequence in $L^2(Q_{\mathcal{A}})$ if and only if there exist $A, B>0$ such that $A \leq |\psi (\gamma)|^2 \leq B$ for a.e. $\gamma\not\in N_{\psi}^{\mathcal{A}}$. In this case the closed span of $\varepsilon_{\mathcal{A}}(\psi)$ is
\[
H_{\psi}^{\mathcal{A}} =\{f\in L^2(Q_{\mathcal{A}}):\ f=0 \ \text{a.e.\ on}\ N_{\psi}^{\mathcal{A}}\},
\]
and $A, B$ are frame bounds for $\varepsilon_{\mathcal{A}}(\psi)$ as a frame for $H_{\psi}^{\mathcal{A}}$.
\end{theorem}
\begin{proof}
Let $A, B$ be frame bounds for $\overline{span}\varepsilon _{\mathcal{A}}(\psi)$. Since  $\varepsilon_{\mathcal{A}}(\psi)$ is  a Bessel sequence, we have $|\psi|^2\leq B$ a.e. Fix any $f\in H_{\psi}^{ \mathcal{A}}$, then  $f\overline{\psi}\in L^2(Q_{\mathcal{A}})$. Since $\psi$ is bounded, so $\varepsilon_{\mathcal{A}}(\psi)$ is a frame for $H_{\psi}^{ \mathcal{A}}$. Now by using the fact that $Q_{(\mathcal{A}^T)^{-1}}$ is an orthonormal basis for $L^2(Q_{\mathcal{A}})$, we have
\begin{align*}
A\int_{Q_{\mathcal{A}}}|f(\gamma)|^2 d\gamma&= A\|f\|^2 _{L^2} \\
&\leq \sum_{k\in \mathbb{Z}^d}|\langle f, \psi e_k\rangle|^2\\
&=\sum_{k\in \mathbb{Z}^d} |\langle f\overline{\psi}, e_k\rangle|^2 \\
&= \|f\overline{\psi}\|^2_{L^2}\\
&=\int_{Q_{\mathcal{A}}} |f(\gamma)|^2 |\psi (\gamma)|^2 d\gamma.
\end{align*}
Since $f$ and $\psi$ both vanish on $N_{\psi}^{\mathcal{A}}$, we have,
\begin{equation}\label{eqkh3.5}
\int_{Q_{\mathcal{A}}\setminus N_{\psi}^{ \mathcal{A}}}|f(\gamma)|^2 (|\psi(\gamma)|^2 -A)d\gamma\geq 0 .
\end{equation}
If $|\psi(\gamma)|^2<A$ on any set $D\subseteq Q_{\mathcal{A}}\setminus N_{\psi}^{ \mathcal{A}}$ of positive measure, then taking $f=\chi_D$ in a inequality \eqref{eqkh3.5},  leads to a contradiction, hence we must have $|\psi (\gamma)|^2 \geq A$ for a.e., $\gamma\not\in N_{\psi}^{ \mathcal{A}}$.
\end{proof}
The next Theorem  shows  the condition under  which,  $\tau_{\mathcal{A}}(f)$ is a frame sequence.
\begin{theorem}\label{thkh3.7}
Let $f\in L^2(\mathbb{R}^d)$ and $A,B>0$.  Then $\tau_{\mathcal{A}}(f)$ is a frame sequence with bounds $A,B$ if and only if, $A\leqslant \Phi_{f}^{\mathcal{A}}(\gamma)\leqslant B$, a.e.  $\gamma \notin N_{\Phi_{f}^{\mathcal{A}}}$, where $N_{\Phi_{f}^{\mathcal{A}}}=\{f\in Q_{\mathcal{A}}:\Phi_{f}^{\mathcal{A}}(\gamma)=0\}$.
\end{theorem}
\begin{proof}
Suppose that $f\in L^2(\mathbb{R}^d)$ and define $V(f):=\overline{span}\{\tau_{\mathcal{A}}(f)\}$. We have to show that, $\tau_{\mathcal{A}}(f)$ is a frame for $V(f)$ if and only if,  for $\gamma\in \mathbb{R}^d$, $\{(E_{-\mathcal{A}k}|\hat{f}|)(\gamma)\}_{k\in \mathbb{Z}^d}$ is  a frame for set $\{F\in L^2(Q_{\mathcal{A}}):F=0$ a.e. on $N_{\Phi_{f}^{\mathcal{A}}}\}$. This happens if and only if $\{(E_{-\mathcal{A}k}|\hat{f}|)(\gamma)\}_{k\in \mathbb{Z}^d}$ is a frame sequence in $L^2(Q_{\mathcal{A}})$.  Now  the result follows  from Theorem \ref{thkh3..3.5}..
\end{proof}
\begin{remark}
With the assumption in Theorem \ref{thkh3.7}., we call the frame generated by $\tau_{\mathcal{A}}(f)$ the frame determined by $\Phi_{f}^{\mathcal{A}}$.
\end{remark}
The following corollary charactrize the members of the subspace generated by a frame sequence of translates.
\begin{corollary}
Assume that $f\in L^2(\mathbb{R}^d)$ and $\tau_{\mathcal{A}}(f)$ is a frame sequence.  Then a function $\psi\in L^2(\mathbb{R}^d)$ belongs to $\overline{span}\{\tau_{\mathcal{A}}(f)\}$ if and only if exists a $\Lambda$-periodic function $F$ whose restriction to $Q_{\mathcal{A}}$, belongs to $L^2(Q_{\mathcal{A}})$, such that $\hat{\psi}=F\hat{f}$.
\end{corollary}
\begin{proof}
Let $\psi\in \overline{span}\{\tau_{\mathcal{A}}(f)\}$, then there exists $\{c_k\}_{k\in \mathbb{Z}^d}\subseteq l^2(Q_{\mathcal{A}})$,  such that $\psi=\sum_{k\in \mathbb{Z}^d} c_k T_{\mathcal{A}k}f$, hence $\hat{\psi}=\sum_{k\in \mathbb{Z}^d} c_k E_{-\mathcal{A}k}\hat{f}$. So $\hat{\psi}=F\cdot \hat{f}$, for $F=\sum_{k\in \mathbb{Z}^d}c_kE_{-\mathcal{A}k}$\\
Conversely, suppose  $\hat{\psi}=F\cdot \hat{f}$,  then $\psi=\hat{F}\cdot f$ and $\psi \big|_{Q_{\mathcal{A}}}$ is $\Lambda$-periodic. So $F\in L^2(Q_{\mathcal{A}})$ and $F=\sum_{k\in \mathbb{Z}^d} c_k E_{-\mathcal{A}k}$.
\end{proof}
At this point,  we recall the fact that a  sequence $\{f_k\}_{k=1}^{\infty}$ in a Hilbert space $\mathcal{H}$ is a Riesz basis for $\mathcal{H}$, if and only if $\{f_k\}_{k=1}^{\infty}$ is a bounded unconditional basis for  $\mathcal{H}$, and that  the functions $e_{k}(\gamma):=e^{-2\pi i(\mathcal{A}^T)^{-1}k\cdot \gamma}$ is an orthonormal basis for $L^2(Q_{\mathcal{A}})$. This fact is applying in the following lemma to characterize when $\varepsilon_{\mathcal{A}}(\psi)$ is a Riesz basis for the Hilbert space $L^2(Q_{\mathcal{A}})$.
\begin{lemma}\label{lekh3.7}
Given $\psi \in L^2(Q_{\mathcal{A}})$, $\varepsilon_{\mathcal{A}}(\psi)$ is an unconditional basis for $L^2(Q_{\mathcal{A}})$ if and only if there exist $A,B>0$ such that $A\leq |\psi (\gamma)|\leq B$ for a.e. $\gamma$. In this case $\varepsilon_{\mathcal{A}}(\psi)$ is a Riesz basis for $L^2(Q_{\mathcal{A}})$.
\end{lemma}
\begin{proof}
Suppose that $\varepsilon_{\mathcal{A}}(\psi)$ is an unconditional basis for $L^2(Q_{\mathcal{A}})$, then,  since $\|\psi e_k\|_{L^2} =\| \psi \|_{L^2}$ for every $k$,  so it  is a  bounded unconditional basis and therefore is a Riesz basis. Since every Riesz basis is an exact frame, so by Theorem \ref{thkh3..3.5}., we must have  $A \leq |\psi (\gamma)|^2 \leq B$, a.e..
\end{proof}
Now by using Lemma \ref{lekh3.7}., we are able to show that $\tau_{\mathcal{A}}(f)$ is a Riesz sequence if and only if $A\leqslant \Phi_{f}^{\mathcal{A}}\leqslant B$ a.e.,  for some $0<A\leq B<\infty$.
\begin{proposition}\label{prkh 3.11}
Let $f\in L^2(\mathbb{R}^d)$ and  $A,B>0$ ,$\tau_{\mathcal{A}}(f)$ is a Riesz sequence with bounds $A,B$ if and only if, $A\leqslant \Phi_{f}^{\mathcal{A}}(\gamma)\leqslant B$, a.e. $\gamma \in Q_{\mathcal{A}}$.
\end{proposition}
\begin{proof}
 Suppose that $\tau_{ \mathcal{A}}(f)$ is an unconditional basis for $L^2(\mathbb{R}^d)$ and $\|\Phi_{f}^{\mathcal{A}}e_k\|_{L^2}=\|\Phi_{f}^{ \mathcal{A}}\|_{L^2}$ for every $k$. By Lemma \ref{lekh3.7}.,  it is a  bounded unconditional basis, and therefore a  Riesz basis, with  $A\leq \Phi_{f}^{ \mathcal{A}}(\gamma)\leq B$ a.e..
\end{proof}
In the next proposition, we determine the Fourier coefficients of the periodic function $\Phi_{f}^{\mathcal{A}}$.
\begin{proposition}\label{prkh5.2}
For  $f\in L^2(\mathbb{R}^d)$, the  Fourier coefficients of $\Phi_{f}^{\mathcal{A}}$ are
\[
c_k=\int_{\mathbb{R}^d}f(\lambda)\overline{f(\lambda+k)}d\lambda,  \quad k\in \mathbb{Z}^d.
\]
\end{proposition}
\begin{proof}
Let $f\in L^2(\mathbb{R}^d)$. Then by using the $\mathbb{Z}^d$-periodicity of $\Phi_{f}^{\mathcal{A}}$, the Fourier coefficients $c_n$ are,
\begin{align*}
c_n&=\int_{[0,1]^d}\Phi_{f}^{\mathcal{A}}(\gamma)\cdot e^{-2\pi i\gamma\cdot n}d\gamma\\
&=\frac{1}{|\det\mathcal{A}|}\int_{[0,1]^d}\sum_{k\in \mathbb{Z}^d}|\hat{f}\big((\mathcal{A}^T)^{-1}(\gamma+k)\big)|^2\cdot e^{-2\pi i\gamma\cdot n}d\gamma\\
&=\frac{1}{|\det \mathcal{A}|} \int_{\mathbb{R}^d}|\hat{f}\big((\mathcal{A}^T)^{-1}(\gamma)\big)|^2\cdot e^{-2\pi i\gamma\cdot n}d\gamma\\
&=\int_{\mathbb{R}^d}|\hat{f}(\gamma)|^2\cdot e^{-2\pi i(\mathcal{A}^T)^{-1}\gamma\cdot n}d\gamma\\
&=\int_{\mathbb{R}^d}\hat{f}(\gamma)\cdot  \overline{(E_{\mathcal{A}k}\hat{f})}(\gamma)d\gamma\\
&=\langle \hat{f}, T_{-\mathcal{A}k}\hat{f}\rangle\\
&=\int_{\mathbb{R}^d}f(\lambda)\overline{T_{-\mathcal{A}k}f(\lambda)}d\lambda\\
&=\int_{\mathbb{R}^d}f(\lambda)\overline{f(\lambda+k)}d\lambda.
\end{align*}
\end{proof}
\begin{remark}\label{prkh3.12}
Note that if $f$ is of compact support, then Proposition \ref{prkh5.2}., implies that only finitely many of the Fourier coefficients of $\Phi_{f}^{\mathcal{A}}$ are non-zero, and therefore $\Phi_{f}^{\mathcal{A}}$ is continuous.
\end{remark}
In this following theorem, we determine conditions under which for $f\in L^2(\mathbb{R}^d)$ with compact support, $\tau_{\mathcal{A}}(f)$ is a Bessel sequence, that cannot be an overcomplete frame sequence. This determine under which for conditions  such $f$, $\tau_{\mathcal{A}}(f)$ is a Riesz sequence.
\begin{theorem}
Assume that $f\in L^2(\mathbb{R}^d)$ has compact support,  then the following hold:
\begin{enumerate}
\item[i)]
$\tau_{\mathcal{A}}(f)$ is a Bessel sequence.
\item[ii)]
$\tau_{\mathcal{A}}(f)$ cannot be an overcomplete frame sequence.
\end{enumerate}
\end{theorem}
\begin{proof}
i)
 By Remark \ref{prkh3.12}., $\Phi_{f}^{\mathcal{A}}$ is continuous. Since $(\mathcal{A}^T)^{-1}[0,1]^d$ is compact, so $\Phi_{f}^{\mathcal{A}}$ is bounded. Therefore there exists $B>0$ such that $|\Phi_{f}^{\mathcal{A}}(\gamma)|\leqslant B$, a.e., consequently, by Proposition \ref{prkh 3.3}., $\tau_{\mathcal{A}}(f)$ is a Bessel sequence.\\ For proving (ii), by Theorem  \ref{thkh3.7}., $\tau_{\mathcal{A}}(f)$ is a frame sequence. Also $\Phi_{f}^{\mathcal{A}}$ is continuous. So Proposition \ref{prkh 3.11}., implies that it is a Riesz sequence, that is, it is a Riesz basis for $\overline{span}\{\tau_{\mathcal{A}}(f)\}$. Thus their members are linearly independent and therefore it is not overcomplete.
\end{proof}
\begin{corollary}
Assume that $f\in L^2(\mathbb{R}^d)$ is compactly supported. Then $\tau_{\mathcal{A}}(f)$ is a Riesz sequence if and only if for every $\gamma\in \mathbb{R}^d$, there exists a $k\in \mathbb{Z}^d$ such that $\hat{f}\big((\mathcal{A}^T)^{-1}(\gamma+k)\big)\neq 0$.
\end{corollary}
\begin{proof}
By Proposition \ref{prkh 3.11}.,  $\Phi_{f}^{\mathcal{A}}$ is continuous. Suppose there exists $\gamma\in \mathbb{R}^d$ such that for every $k\in \mathbb{Z}^d$, $\hat{f}\big((\mathcal{A}^T)^{-1}(\gamma+k)\big)=0$. Then $\Phi_{f}^{\mathcal{A}}=0$, which contradicts Proposition \ref{prkh 3.11}. Thus $\hat{f}\big((\mathcal{A}^T)^{-1}(\gamma+k)\big)\neq 0$.\\
For the converse, assume that for any $\gamma\in \mathbb{R}^d$, $\hat{f}\big((\mathcal{A}^T)^{-1}(\gamma+k)\big)\neq 0$. This implies that $\tau_{\mathcal{A}}(f)$ is a Riesz sequence. Now, $\Phi_{f}^{\mathcal{A}}(\gamma)\neq 0$, implies that $\Phi_{f}^{\mathcal{A}}$ is continuous and bounded. Thus $\Phi_{f}^{\mathcal{A}}(\gamma)>0$, that is, it takes its  minimum, which is strictly positive. Therefore, it is a Riesz sequence which is bounded from below.
\end{proof}
In what follows, by using Lemma \ref{lekh3.14}. and Proposition \ref{lekh3.15}., we determine conditions under which $\{T_{\mathcal{A}k}f+T_{\mathcal{A}k+n}f\}_{k \in \mathbb{Z}^d}$ is a frame for $\overline{span}\{\tau_{\mathcal{A}}(f)\}$.
\begin{lemma}\label{lekh3.14}
Let $\{\psi_k\}_{k\in \mathbb{Z}^d}$ is a Bessel sequence and complete in  Hilbert space $\mathcal{H}$ and $n_{0}\in\mathbb{Z}^d$ is given,  then $\{\psi_k+\psi_{k+n_0}\}_{k\in \mathbb{Z}^d}$ is complete.
\end{lemma}
\begin{proof}
We proof the lemma for the case $d=1$ and  $n_{0}=1$.  The general case is almost the same, suppose that $\psi\in \mathcal{H}$ is arbitrary such that for every $k$, $\langle \psi,\psi_{k}+\psi_{k+1}\rangle=0$. We show that $\psi=0$.   Indeed $\langle \psi,\psi_{k}+\psi_{k+1}\rangle=0 $,  implies that $\langle \psi, \psi_k\rangle =-\langle \psi, \psi_{k+1}\rangle$. Then for every $k$, $|\langle\psi, \psi_k\rangle |=|\langle \psi, \psi_{k+1}\rangle |$, hence $|\langle \psi, \psi_k \rangle |$ is constant. On the other hand, by our assumption,  $\{\psi_k\}_{k\in \mathbb{Z}}$ is a Bessel sequence. So,
\[
  \sum_{k\in \mathbb{Z}}|\langle \psi, \psi_k \rangle |^2\leqslant B \|\psi\|^2,
  \]
   for some  $B>0$. Hence  $\langle \psi, \psi_k \rangle =0$,
   for all $k$. Thus $\psi =0$, since $\{\psi_k\}_{k\in \mathbb{Z}}$ is complete.
\end{proof}
\begin{proposition}\label{lekh3.15}
For  $f \in L^2(\mathbb{R}^d)$ let  $\tilde{f}=f+T_{\mathcal{A}n}f$. Then $\Phi_{\tilde{f}}^{\mathcal{A}}(\gamma)=|1+e^{-2\pi i (\mathcal{A}^T)^{-1}n\cdot \gamma}|^2\Phi_{f}^{\mathcal{A}}(\gamma)$,  and $\Phi_{\tilde{f}}^{\mathcal{A}}(\gamma)$ determine a  frame for $V(f):=\overline{span}\{\tau_{\mathcal{A}}(f)\}$ if and only if $\Phi_{f}^{\mathcal{A}}(\gamma)$ determines a  frame for $V(f)$ and $e^{-2\pi i (\mathcal{A}^T)^{-1}n\cdot  \gamma}\neq -1$.
\end{proposition}
\begin{proof}
By the definition of $\Phi_{\tilde{f}}^{\mathcal{A}}$, we have
\begin{align*}
\Phi_{\tilde{f}}^{\mathcal{A}}(\gamma)&=\frac{1}{|\det \mathcal{A}|}\sum_{k \in \mathbb{Z}^d}|\mathcal{F}\tilde{f}\big((\mathcal{A}^T)^{-1}(\gamma+k)\big)|^2\\
&=\frac{1}{|\det \mathcal{A}|}\sum_{k \in \mathbb{Z}^d}|(\mathcal{F}f+\mathcal{F}T_{\mathcal{A}n}f)\big((\mathcal{A}^T)^{-1}(\gamma+k)\big)|^2\\
&=\frac{1}{|\det \mathcal{A}|}\sum_{k \in \mathbb{Z}^d}|\hat{f}\big((\mathcal{A}^T)^{-1}(\gamma+k)\big)+E_{-\mathcal{A}n}\hat{f}\big((\mathcal{A}^T)^{-1}(\gamma+k)\big)|^2\\
&=\frac{1}{|\det \mathcal{A}|}\sum_{k \in \mathbb{Z}^d}|\hat{f}\big((\mathcal{A}^T)^{-1}(\gamma+k)\big)\\
&\quad +e^{-2\pi i(\mathcal{A}^T)^{-1}n\cdot (\gamma+k)}\hat{f}\big((\mathcal{A}^T)^{-1}(\gamma+k)\big)|^2\\
&=\frac{1}{|\det \mathcal{A}|}\sum_{k \in \mathbb{Z}^d}|1+e^{-2\pi i(\mathcal{A}^T)^{-1}n\cdot (\gamma+k)}|^2\cdot |\hat{f}\big((\mathcal{A}^T)^{-1}(\gamma+k)\big)|^2\\
&=|1+e^{-2\pi in\cdot (\mathcal{A}^T)^{-1}\gamma}|^2\Phi_{f}^{\mathcal{A}}(\gamma).
\end{align*}
Therefore by Theorem \ref{thkh3..3.5}., there exist $A,B>0$ such that $A\leqslant\Phi_{\tilde{f}}^{\mathcal{A}}(\gamma)\leqslant B$, a.e., $\gamma \in Q_{\mathcal{A}}\setminus N$ where\\
$N=\{\gamma \in Q_{\mathcal{A}}: \Phi_{\tilde{f}}^{\mathcal{A}}(\gamma)=0\}$
\end{proof}
\begin{corollary}
Let $f \in L^2(\mathbb{R}^d)$ and assume that $\tau_{\mathcal{A}}(f)$ is a frame for $V:=\overline{span}\{\tau_{\mathcal{A}}(f)\}$, $n \in \mathbb{Z}^d$. Then $\{T_{\mathcal{A}k}f+T_{\mathcal{A}k+n}f\}_{k \in \mathbb{Z}^d}$ is a frame for $V$,  if and only if $2(\mathcal{A}^T)^{-1}\gamma\cdot n$ is not an odd number,  for $\gamma\in [0,1]^d$.
\end{corollary}
\begin{proof}
It is an immediate consequence of  Lemma \ref{lekh3.14}. and Lemma  \ref{lekh3.15}..
\end{proof}
We conclude,  with an example that $\tau_{\mathcal{A}}(f)$ is a Parseval frame, but  it is not a Riesz sequences.
\begin{example}
Let $f\in L^2(\mathbb{R}^d)$ and $\hat{f}(\gamma)=\chi_{(\mathcal{A}^T)^{-1}}\big[-\frac13,\frac13\big]^d(\gamma)$
for $\gamma\in \mathbb{R}^d$ then
\begin{align*}
\Phi_{f}^{\mathcal{A}}(\gamma)&= \frac{1}{|\det \mathcal{A}|}\sum_{k\in \mathbb{Z}^d}\big|\chi_{(\mathcal{A}^T)^{-1}[-\frac13,\frac13]^d}\big((\mathcal{A}^T)^{-1}(\gamma+k)\big)\big|^2\\2
&=\chi_{(\mathcal{A}^T)^{-1}}\big[-\frac13,\frac13\big]^d,
\end{align*}
and by Theorem \ref{thkh3.7}., $\tau_{\mathcal{A}}(f)$ is a frame sequence with frame bounds $A=B=1$, but does not form a Riesz sequence.
\end{example}

\bibliographystyle{amsplain}

\end{document}